\documentclass[12pt]{amsart}
\usepackage{ucs}
\usepackage[english]{babel}
\usepackage[utf8x]{inputenc} 
\usepackage{amsfonts}

\usepackage{url}
\usepackage{amssymb}
\usepackage{amsmath}
\usepackage{amsrefs}
\usepackage{nccmath}
\usepackage{mathtools}
\usepackage{bbm}
\usepackage{hyperref}
\hypersetup{
  colorlinks=true,
  linkcolor=blue,
  citecolor=red
}

\newtheorem{theorem}{Theorem}
\newtheorem*{theorem*}{Theorem}
\newtheorem{lem}[theorem]{Lemma}
\newtheorem{proposition}[theorem]{Proposition}

\newtheorem{corollary}[theorem]{Corollary}

\def\N{\mathbb{N}}
\def\R{\mathbb{R}}
\def\dd{\mathrm{d}}
\def\ee{\mathrm{e}}

\begin{document}
\title{The speed of sequential asymptotic learning} \author{Wade
  Hann-Caruthers, Vadim V. Martynov and Omer Tamuz}

\address{California Institute of Technology}

\thanks{The authors would like to thank Christophe Chamley, Gil
  Refael, Peter S{\o}rensen, Philipp Strack, Ye Wang, Ivo Welch and Leeat Yariv for helpful comments
  and discussions. This
    work was supported by  a grant from the Simons Foundation (\#419427,  Omer Tamuz)}

\maketitle

\begin{abstract}
  In the classical herding literature, agents receive a private signal
  regarding a binary state of nature, and sequentially choose an
  action, after observing the actions of their predecessors. When the
  informativeness of private signals is unbounded, it is known that
  agents converge to the correct action and correct belief. We study
  how quickly convergence occurs, and show that it happens more slowly
  than it does when agents observe signals. However, we also show that
  the speed of learning from actions can be arbitrarily close to the
  speed of learning from signals. In particular, the expected time until the agents stop taking the wrong action can be either finite or infinite, depending on the private signal distribution.  In the canonical case of Gaussian
  private signals we calculate the speed of convergence precisely, and
  show explicitly that, in this case, learning from actions is
  significantly slower than learning from signals.
\end{abstract}

\section{Introduction}

When making decisions, we often rely on the decisions that others
before us have made. Sequential learning models have been used to
understand different phenomena that occur when many individuals make
decisions based on the observed actions of others. These include herd
behavior (cf.~\cite{banerjee1992simple}), where many agents make the
same choice, as well as informational cascades
(e.g.~\cite{bikhchandani1992theory}), where the actions of the first
few agents provide such compelling evidence that later agents no
longer have incentive to consider their own private information.

Such results on how information aggregation can fail are complemented
by results which demonstrate that when private signals are arbitrarily
strong, learning is robust to this kind of
collapse~\cite{smith2000pathological}. In particular, in a process
called asymptotic learning (see, e.g.,~\cite{acemoglu2011bayesian}),
agents will eventually choose the correct action and their beliefs
will converge to the truth. A question that has not been answered in
the literature is: how quickly does this happen? And how does the
speed of learning compare to a setting in which agents observe signals
rather than actions?

We consider the classical setting of a binary state of nature and
binary actions, where each of the two actions is optimal at one of the
states. The agents receive private signals that are
independent conditioned on the state. These signals are unbounded, in
the sense that an agent's posterior belief regarding the state can be
arbitrarily close to both 0 and 1. The agents are exogenously ordered,
and, at each time period, a single agent takes an action, after
observing the actions of her predecessors.

We measure the speed of learning by studying how the public belief
evolves as more and more agents act. Consider an outside observer who
observes the actions of the sequence of agents. The public belief is
the posterior belief that such an outside observer assigns to the
correct state of nature. It provides a measure of how well the
population has learned the state. Since signals are unbounded, the
public belief tends to $1$ over time~\cite{smith2000pathological}; equivalently, the corresponding
log-likelihood ratio tends to infinity. As the outside observer may also be interested in learning the state, it is natural to ask how quickly she converges to the correct belief, and, in particular, to understand her asymptotic speed of learning when observing actions. Asymptotic rates of convergence are an important tool in the study of inference processes in statistical theory, and have also been studied in social learning models in the Economics literature (e.g.,~\cite{vives1993fast, duffie2007information, duffie2009information}).

When agents observe the {\em signals} (rather than actions) of all of
their predecessors, this log-likelihood ratio is asymptotically
linear. Thus, it cannot grow faster than linearly when the agents
observe actions. Our first main finding is that when observing
actions, the log-likelihood ratio always grows
sub-linearly. Equivalently, the public belief converges
sub-exponentially to 1. Our second main finding is that, depending on
the choice of private signal distributions, the log-likelihood ratio
can grow at a rate that is arbitrarily close to linear.

We next analyze the specific canonical case of Gaussian private
signals. Here we calculate precisely the asymptotic behavior of the
log-likelihood ratio of the public belief. We show that learning from
actions is significantly slower than learning from signals: the
log-likelihood ratio behaves asymptotically as $\sqrt{\log t}$. To
calculate this we develop a technique that allows, much more
generally, for the long-term evolution of the public belief to be
calculated for a large class of signal distributions.

Since, in our setting of unbounded signals, agents eventually take the
correct action, an additional, natural measure of the speed of
learning is the expected time at which this happens: how long does it
take until no more mistakes are made? We call this the {\em time to
  learn}.

We show that the expected time to learn depends crucially on the
signal distributions. For distributions, such as the Gaussian, in
which strong signals occur with very small probability, we show that
the expected time to learn is infinite.\footnote{In the benchmark case
  of observed signals this time is finite, for any signal
  distribution.}  However, when strong signals are less rare, this
expectation is finite.\footnote{This result disproves a conjecture of
  S{\o}rensen~\cite[page 36]{sorensen1996rational}.} Intuitively, when
strong signals are rare, agents are more likely to emulate their
predecessors, and so it may take a long time for a mistake to be
corrected.

Finally, in the Gaussian case, we study another measure of the speed
of learning.  Namely, we consider directly how the probability of
choosing the incorrect action varies as agents see more and more of
the other agents' decisions before making their own. We find that this
probability is asymptotically no less than $1/t^{1+\varepsilon}$ for
any $\varepsilon>0$. In contrast, when agents can observe the private
signals of their predecessors, the probability of mistake decays
exponentially, and so also in this sense learning from signals is much
faster than learning from actions.

\subsection{Related literature}

Several previous studies have considered the same
question. Chamley~\cite{chamley2004rational} gives an estimate for the
evolution of the public belief for a class of private signal
distributions with fat tails. He also studies the speed of convergence
in the Gaussian case using a computer
simulation. S{\o}rensen~\cite[Lemma 1.9]{sorensen1996rational} has
published a claim related to our Theorem~\ref{thm:sublinear}, with an
unfinished proof. Also in~\cite{sorensen1996rational}, S{\o}rensen
shows that the expected time to learn is infinite for some signal
distributions, and conjectures that it is always infinite, which we
show to not be true. In~\cite{smith96pathological}, an early version of~\cite{smith2000pathological}, the question of the time to learn is also addressed, and an example is given in which the time to learn is infinite, but is finite conditioned on one of the states. A concurrent paper by Rosenberg and
Vieille~\cite{rosenberg2017efficiency} studies related questions. In
particular they study the time until the first correct action, as well
as the number of incorrect actions---which are related to our time to
learn---and characterize when they have finite expectations.

A related model is studied by Lobel, Acemoglu, Dahleh and
Ozdaglar~\cite{lobel2009rate}, who consider agents who also act
sequentially, but do not observe all of their predecessors'
actions. They study how the speed of learning varies with the network
structure. Vives~\cite{vives1993fast}, in a paper with a very similar
spirit to ours, studies the speed of sequential learning in a model
with actions chosen from a continuum, and where agents observe a noisy
signal about their predecessors' actions. He similarly shows that learning
is significantly slower than in the benchmark case. An overview
of this literature is given by Vives in his book~\cite[Chapter
6]{vives2010information}.

\section{Model}
Let $\theta \in \{-1, +1\}$ be the true state of the world, with each
state a priori equally likely\footnote{We make this simplification of
  a (1/2,1/2) prior to reduce the complexity of the presentation, but
  all results hold for general priors.}. Each rational agent
$t \in \{1,2,\ldots\}$ receives a private signal $s_t$. The signals
are i.i.d.\ conditioned on $\theta$: if $\theta=+1$ they have
cumulative distribution function (CDF) $F_+$ and if $\theta=-1$ they
have CDF $F_-$.\footnote{One could consider signals that take values
  in a general measurable space (rather than $\R$), but the choice of
  $\R$ is in fact without loss of generality, since all standard
  measurable spaces are isomorphic.} We assume that $F_+$ and $F_-$
are absolutely continuous with respect to each other, so that private
signals never completely reveal the state.

Let
$$L_t = \log{\frac{\mathbb{P}(\theta = +1|s_t )}{\mathbb{P}(
    \theta = -1 | s_t)}}$$
be the log-likelihood ratio of the belief induced by the agent's
private signal. We assume that private signals are unbounded, in the sense that $L_t$ is unbounded: for every $M \in \R$ the probability that $L_t > M$ is positive, as is the probability that $L_t< -M$. We denote by $G_+$ and $G_-$ the conditional CDFs of $L_t$.

The agents act sequentially, with agent $t$ acting after observing
the actions of agents $\{1,\ldots,t-1\}$. The utility of the action
$a_t \in \{-1,+1\}$ is 1 if $a_t=\theta$ and 0 otherwise.

Denote the public belief by
\begin{align*}
  \mu_t = \mathbb{P}(\theta=+1 | a_1, \dots, a_{t-1}).
\end{align*}
This is the posterior held by an outside observer after recording the
actions of the first $t-1$ agents.  We denote by $\ell_t$ the log-likelihood ratio of the public belief:
\begin{align*}
  \ell_t = \log{\frac{\mu_t}{1 - \mu_t}}\cdot
\end{align*}
In equilibrium, agent $t$ 
chooses $a_t=+1$ iff\footnote{For simplicity, we assume that agents choose action $-1$ when indifferent. This will have no impact on our results.}
\begin{align*}
  \log{\frac{\mathbb{P}(\theta = +1 | a_1, \dots, a_{t-1}, s_t)}{\mathbb{P}(\theta = -1 | a_1, \dots, a_{t-1}, s_t)}} > 0.
\end{align*}
A simple calculation shows that this occurs iff
$$\ell_t + L_t > 0.$$

Now, another straightforward calculation shows that when $a_t=+1$,
\begin{align}
  \label{eq:ell_t}
  \ell_{t+1} = \ell_t + D_+(\ell_t),
\end{align}
where
\begin{align*}
    D_+(x) = \log{\frac{1-G_{+}(-x)}{1-G_{-}(-x)}}\cdot
\end{align*}
Likewise, when $a_t=-1$,   
$$
\ell_{t+1} = \ell_t + D_-(\ell_t),
$$
where
\begin{align*}
    D_-(x) = \log{\frac{G_{+}(-x)}{G_{-}(-x)}}\cdot
\end{align*}
We can interpret $D_+(\ell_t)$ and $D_-(\ell_t)$ as the contributions of agent $t$'s action to the public belief.

\section{The evolution of public belief}
Consider a baseline model, in which each agent observes the private signals of all of her predecessors. In this case the public log-likelihood ratio $\tilde{\ell}_t$ would equal the sum 
\begin{align*}
    \tilde{\ell}_t = \sum_{\tau=1}^t L_\tau.
\end{align*}
Conditioned on the state this is the sum of i.i.d.\ random variables,
and so by the law of large numbers we have that the limit
$\lim_t\tilde{\ell}_t/t$ would---conditioned on (say)
$\theta=+1$---equal the conditional expectation of $L_t$, which is
positive.\footnote{In fact, $\mathbb{E}(L_t|\theta=+1)$ is equal to
  the Kullback-Leibler divergence between $F_+$ and $F_-$, which is
  positive as long as the two distributions are different.}

\subsection{Sub-linear public beliefs}

Our first main result shows that when agents observe actions rather than signals, the public log-likelihood ratio grows sub-linearly, and so learning from actions is always slower than learning from signals.
\begin{theorem}
\label{thm:sublinear}
It holds with probability 1 that $\lim_t \ell_t/t=0$.
\end{theorem}

Our second main result shows that, depending on the choice of private signal distributions, $\ell_t$ can grow at a rate that is arbitrarily close to linear: given any sub-linear function $r_t$, it is possible to find private signal distributions so that $\ell_t$ grows as fast as $r_t$.
\begin{theorem}
\label{thm:fast-sublinear}
For any $r \colon \N \to \R_{>0}$ such that $\lim_t r_t/t = 0$ there
exists a choice of CDFs $F_-$ and $F_+$ such that
\begin{align*}
  \liminf_{t \to \infty}\frac{|\ell_t|}{r_t} > 0
\end{align*}
with probability 1.
\end{theorem}
For example, for some choice of private signal distributions, $\ell_t$ grows asymptotically at least as fast as $t / \log t$, which is sub-linear but (perhaps) close to linear.

\subsection{Long-term behavior of public beliefs}
We next turn to estimating more precisely the long-term behavior of
the public log-likelihood ratio $\ell_t$. Since signals are unbounded, agents learn the state, so
that $\ell_t$ tends to $+\infty$ if $\theta=+1$, and to $-\infty$ if
$\theta=-1$. In particular $\ell_t$ stops changing sign from some $t$
on, with probability 1; all later agents choose the correct action.

We consider without loss of generality the case that $\theta=+1$, so
that $\ell_t$ is positive from some $t$ on. Thus,
recalling~\eqref{eq:ell_t}, we have that from some $t$ on,
\begin{align*}
  \ell_{t+1} = \ell_t + D_+(\ell_t).
\end{align*}
This is the recurrence relation that we need to solve in order to
understand the long term evolution of $\ell_t$. To this end, we
consider the corresponding differential equation:
$$
\frac {\dd f}{\dd t}(t)=D_+(f(t)).
$$
Recall that $G_-$ is the CDF of the private log-likelihood ratio $L_t$, conditioned on $\theta=-1$. We show (Lemma~\ref{lem:D}) that $D_+(x)$ is well approximated by
$G_-(-x)$ for high $x$, in the sense that
\begin{align*}
  \lim_{x \to \infty}\frac{D_+(x)}{G_-(-x)} = 1.
\end{align*}
In some applications (including the Gaussian one, which we consider
below), the expression for $G_-$ is simpler than that for $D_+$, and
so one can instead consider the differential equation
\begin{align}
  \label{eq:diff}
  \frac {\dd f}{\dd t}(t)=G_-(-f(t)).
\end{align}
This equation can be solved analytically in many cases in which $G_-$
has a simple form. For example, if $G_-(-x) = \ee^{-x}$ then
$f(t) = \log(t+c)$, and if $G_-(-x) = x^{-k}$ then
$f(t) = ((k+1) \cdot t + c)^{1/(k+1)}$.

We show that solutions to this equation have the same long term
behavior as $\ell_t$, given that $G_-$ satisfies some regularity conditions.
\begin{theorem}
  \label{thm:diff}
  Suppose that $G_-$ and $G_+$ are continuous, and that the left tail
  of $G_-$ is convex and differentiable. Suppose also that
  $f \colon \mathbb{R}_{>0} \rightarrow \mathbb{R}_{>0}$ satisfies
  \begin{align}
  \label{eq:diff equation}
    \frac{\dd f}{\dd t}(t) = G_-(-f(t))
  \end{align}
  for all sufficiently large $t$. Then conditional on $\theta = +1$,
  \begin{align*}
    \lim_{t \rightarrow \infty}{\frac{\ell_t}{f(t)}} = 1
  \end{align*}
  with probability $1$.
\end{theorem}
The condition\footnote{By ``the left tail of $G_-$ is convex and
  differentiable'' we mean that there is some $x_0$ such that,
  restricted to $(-\infty,x_0)$, $G_-$ is convex and
  differentiable.} on $G_-$ is satisfied when the random variables
$L_t$ (i.e., the log-likelihood ratios associated with the private
signals), conditioned on $\theta=-1$, have a distribution with a
probability density function that is monotone decreasing for all $x$
less than some $x_0$. This is the case for the normal distribution,
and for practically every non-atomic distribution one may encounter in
the standard probability and statistics literatures.

\subsubsection{Gaussian signals}
In the Gaussian case, $F_+$ is Normal with mean $+1$ and variance
$\sigma^2$, and $F_-$ is Normal with mean $-1$ and the same
variance. A simple calculation shows that $G_-$ is the Gaussian
cumulative distribution function, and so we cannot solve the
differential equation~\eqref{eq:diff} analytically. However, we can
bound $G_-(x)$ from above and from below by functions of the form
$\ee^{-c \cdot x^2}/x$. For these functions the solution
to~\eqref{eq:diff} is of the form $f(t)=\sqrt{\log t}$, and so we can
use Theorem~\ref{thm:diff} to deduce the following.
\begin{theorem}
\label{l_t_asymptotics}
When private signals are Gaussian, then conditioned on
$\theta = + 1$, 
  \begin{align*}
    \lim_{t \to \infty}\frac{\ell_t}{(2\sqrt{2}/\sigma)\cdot\sqrt{\log t}} = 1
  \end{align*}
  with probability 1.
\end{theorem}
Recall, that when private signals are observed, the public
log-likelihood ratio $\ell_t$ is asymptotically {\em linear}.  Thus,
learning from actions is far slower than learning from signals in the
Gaussian case.

\subsection{The expected time to learn}
\label{subsec:expected_time_to_learn}
When private signals are unbounded then with probability 1 the agents
eventually all choose the correct action $a_t=\theta$. A natural
question is: how long does it take for that to happen? Formally, we
define the {\em time to learn}
\begin{align*}
  T_L = \min\{t \,:\, a_{\tau} = \theta \mbox{ for all } \tau \geq t\},
\end{align*}
and study its expectation. Note that in the baseline case of observed
signals $T_L$ has finite expectation, since the probability of a
mistake at time $t$ decays exponentially with $t$.

We first study the expectation of $T_L$ in the case of Gaussian
signals. To this end we define the {\em time of first mistake} by
\begin{align*}
  T_1 = \min\{t \,:\, a_t \neq \theta\}
\end{align*}
if $a_t \neq \theta$ for some $t$, and by $T_1 = 0$ otherwise. We
calculate a lower bound for the distribution of $T_1$, showing that it
decays at most as fast as $1/t$.
\begin{theorem}
  \label{probability_of_mistake}
  When private signals are Gaussian then for every $\varepsilon > 0$
  there exists a $k > 0$ such that for all $t$
  \begin{align*}
    \mathbb{P}(T_1 = t) \geq \frac{k}{t^{1+\varepsilon}}\cdot
  \end{align*}
\end{theorem}
Thus $T_1$ has a very thick tail, decaying far slower than the
exponential decay of the baseline case. In particular, $T_1$ has
infinite expectation, and so, since $T_L > T_1$, the expectation of
the time to learn $T_L$ is also infinite.

In contrast, we show that when private signals have thick
tails---that is, when the probability of a strong signal vanishes
slowly enough---then the time to learn has finite expectation. In
particular, we show this when the left tail of $G_-$ and the right tail of $G_+$ are polynomial.\footnote{Recall that $G_-$ is the conditional cumulative distribution function of the private log-likelihood ratios $L_t$.} 
\begin{theorem}
\label{thm:polytail-finite-expectation}
  Assume that $G_-(-x) = c \cdot x^{-k}$ and that $G_+(x) = 1-c \cdot x^{-k}$ for some $c > 0$ and $k > 0$, and for all $x$ greater than some $x_0$. Then
  $\mathbb{E}(T_L) < \infty$.
\end{theorem}
An example of private signal distributions $F_+$ and $F_-$ for which $G_-$ and $G_+$ have this form is given by the probability density functions
\begin{align*}
  f_-(x) = \begin{cases}
    c \cdot \ee^{-x}x^{-k-1}&\mbox{when } 1 \leq x \\
    0               &\mbox{when } -1 < x < 1\\
    c \cdot (-x)^{-k-1}&\mbox{when } x \leq -1.
    \end{cases}
\end{align*}
and $f_+(x) = f_-(-x)$, for an appropriate choice of normalizing
constant $c>0$. In this case $G_-(-x) = 1-G_+(x) = \frac{c}{k}x^{-k}$
for all $x>1$.\footnote{Theorem~\ref{thm:polytail-finite-expectation}
  can be proved for other thick-tailed private signal distributions:
  for example, one could take different values of $c$ and $k$ for
  $G_-$ and $G_+$, or one could replace their thick polynomial tails
  by even thicker logarithmic tails. For the sake of readability we
  choose to focus on this case.}

The proof of Theorem~\ref{thm:polytail-finite-expectation} is rather technically involved, and we provide here a rough sketch of the ideas behind it. 

We say that there is an {\em upset} at time $t$ if $a_{t-1} \neq
a_t$. We denote by $\Xi$ the random variable which assigns to each
outcome the total number of upsets
\begin{align*}
  \Xi = |\{t\,:\,a_{t-1} \neq a_t\}|.
\end{align*}
We say that there is a {\em run} of length $m$ from time $t$ if
$a_t = a_{t+1} = \cdots = a_{t+m-1}$. As we will
condition on $\theta=+1$ in our analysis, we say that a run from time
$t$ is {\em good} if $a_t = 1$ and {\em bad} otherwise. A trivial but
important observation is that the number of maximal finite runs is equal to the number of upsets, and so, if $\Xi = n$, and if $T_L=t$, then there
is at least one run of length $t/n$ before time $t$. Qualitatively,
this implies that if the number of upsets is small, and if the time to
learn is large, then there is at least one long run before the time to
learn.

We show that it is indeed unlikely that $\Xi$ is large: the
distribution of $\Xi$ has an exponential tail. Incidentally, this
holds for {\em any} private signal distribution:
\begin{proposition}
\label{prop:exponential_decay_of_runs}
  For every private signal distribution there exist $c > 0$ and
  $0 < \gamma < 1$ such that for all $n > 0$
  \begin{align*}
    \mathbb{P}(\Xi \geq n) \leq c \gamma^n.
  \end{align*}
\end{proposition}
Intuitively, this holds because whenever an agent takes the correct
action, there is a non-vanishing probability that all subsequent agents
will also do so, and no more upsets will occur.

Thus, it is very unlikely that the number of upsets $\Xi$ is large. As
we observe above, when $\Xi$ is small then the time to learn $T_L$ can
only be large if at least one of the runs is long. When $G_-$ has a
thin tail then this is possible; indeed,
Theorem~\ref{probability_of_mistake} shows that the first finite run has
infinite expected length when private signals are Gaussian. However,
when $G_-$ has a thick, polynomial left tail of order $x^{-k}$, we
show that it is very unlikely for any run to be long: the probability
that there is a run of length $n$ decays at least as fast as
$\exp(-n^{k/(k+1)})$, and in particular runs have finite expected
length. Intuitively, when strong signals are rare then runs tend to be
long, as agents are likely to emulate their predecessor. Conversely,
when strong signals are more likely then agents are more likely to
break a run, and so runs tend to be shorter.

Putting together these insights, we conclude that it is unlikely that
there are many runs, and, in the polynomial signal case, it is
unlikely that runs are long. Thus $T_L$ has finite expectation.

\subsection{Probability of taking the wrong action}
Yet another natural metric of the speed of learning is the probability
of mistake
\begin{align*}
 p_t =\mathbb{P}(a_t\neq \theta).
\end{align*}
Calculating the asymptotic behavior of $p_t$ seems harder to tackle. 

For the Gaussian case, while we cannot estimate $p_t$ precisely, Theorem~\ref{probability_of_mistake} immediately implies a lower bound: $p_t$ is at least $k / t^{1+\varepsilon}$, for every $\varepsilon>0$ and $k$ that depends on $\varepsilon$. This is much larger than the exponentially vanishing probability of mistake in the revealed signal baseline case.

More generally, we can use Theorem~\ref{thm:sublinear} to show that
$p_t$ vanishes sub-exponentially for any signal distribution, in the sense that 
\begin{align*}
    \lim_{t \to \infty} \frac{1}{t}\log p_t = 0.
\end{align*}
To see this, note that the probability of mistake at time $t-1$, conditioned on the observed actions, is exactly equal to 
\begin{align*}
    \min\{\mu_t,1-\mu_t\};
\end{align*}
where we recall that 
\begin{align*}
    \mu_t= \mathbb{P}(\theta=+1 | a_1, \dots, a_{t-1}) = \frac{\ee^{\ell_t}}{\ee^{\ell_t}+1} 
\end{align*} 
is the public belief. This is due to the fact that if the outside observer, who holds belief $\mu_t$, had to choose an action, she would choose $a_{t-1}$, the action of the last player she observed, a player who has strictly more information than her. Thus
\begin{align*}
    p_t = \mathbb{E}(\min\{\mu_t,1-\mu_t\}) = \mathbb{E}\left(\frac{1}{\ee^{|\ell_t|}+1}\right),
\end{align*}
and since, by Theorem~\ref{thm:sublinear}, $|\ell_t|$ is sub-linear, it follows that $p_t$ is sub-exponential.

\section{Conclusion}
In this paper we consider a classical setting of sequential asymptotic learning from actions of others. We show that learning from actions is slow, as compared to the speed of learning when observing others' private signals, in the sense that the public log-likelihood ratio tends more slowly to infinity. However, it is possible to approach the linear rate of learning from signals and achieve any sub-linear rate.

We calculate the speed of learning precisely in the case of Normal private signals (among a large class of private signal distributions) and show that learning is very slow. We also show that in the Gaussian case the expected time to learn is infinite, as opposed to cases of more thick-tailed distributions, in which it is finite.

For the Gaussian case we also provide a lower bound for the probability of mistake. Finding a matching upper bound seems beyond our reach at the moment, and provides a compelling open problem for further research.

\bibliographystyle{abbrv}
\bibliography{main}
\newpage

\appendix

\section{Sub-linear learning}
Before proving our main theorems we make the observation (which has appeared before, e.g.,~\cite{chamley2004rational}) that {\em the log-likelihood ratio of the log-likelihood ratio is the log-likelihood ratio}. Formally, if $\nu_+$ and $\nu_-$ are the conditional distributions of the private log-likelihood ratio $L_t$ (i.e., have CDFs $G_+$ and $G_-$), then
\begin{align*}
    \log\frac{\dd\nu_+}{\dd\nu_-}(x) = x.
\end{align*}
It follows that
\begin{align}
   \label{eq:nu}
    G_+(x) = \int_{-\infty}^x \,\dd\nu_+(\zeta) = \int_{-\infty}^x \ee^\zeta \, \dd\nu_-(\zeta).
\end{align}

Our first lemma shows that asymptotically, $D_+$ behaves like the left tail of $G_-$, and $D_-$ behaves like the right tail of $G_+$.
\begin{lem}
\label{lem:D}
\begin{align*}
    \lim_{x \rightarrow \infty}{\frac{D_{+}(x)}{G_{-}(-x)}} = 1 \text{ and } \lim_{x \rightarrow -\infty}{\frac{D_{-}(x)}{G_{+}(-x)-1}} = 1.
\end{align*}
\end{lem}
\begin{proof}
By definition,
\begin{align*}
    D_+(x) = \log\frac{1-G_+(-x)}{1-G_-(-x)}.
\end{align*}
Since $\log(1-z)=-z+O(z^2)$, it holds for all $x$ large enough that
\begin{align*}
    D_+(x) > G_-(-x) - 2\cdot G_+(-x).
\end{align*}
Applying~\eqref{eq:nu} yields
\begin{align*}
    D_+(x) > \int_{-\infty}^{-x}(1-2\ee^\zeta)\,\dd\nu_-(\zeta),
\end{align*}
and so for any $\varepsilon$ and all $x$ large enough,
\begin{align*}
    D_+(x) > (1-\epsilon)\cdot \int_{-\infty}^{-x}\,\dd\nu_-(\zeta) = (1-\epsilon)G_-(-x).
\end{align*}
Using the same approximation of the logarithm, we have that
\begin{align*}
    D_+(x) < (1+\epsilon) G_-(-x)-G_+(-x) < (1+\epsilon) G_-(-x).
\end{align*}
The statement for $D_+$ now follows by taking $\varepsilon$ to zero.
The corresponding bounds on $D_-$ follow by identical arguments.
\end{proof}

\begin{proof}[Proof of Theorem~\ref{thm:sublinear}]
Condition on $\theta = +1$. Then $\ell_t$ is with probability 1 positive from some point on, and all agents take action $+1$ from this point on. Hence, for all $t$ large enough,
\begin{align*}
    \ell_{t+1} = \ell_t + D_+(\ell_t).
\end{align*}
By Lemma~\ref{lem:D}, we know that $\lim_x D_+(x)=0$. Hence for every $\epsilon > 0$ and all $t$ large enough, $|\ell_{t+1}-\ell_t| < \epsilon$. It follows that the limit $\lim_t \ell_t/t=0$. The analysis of the case $\theta=-1$ is identical.
\end{proof}

\begin{proof}[Proof of Theorem~\ref{thm:fast-sublinear}]
  Given $r_t$, we will construct private signal distributions such that $\liminf_t |\ell_t|/r_t > 0$ with probability one. These distributions will furthermore have the property that $D_+(x) = -D_-(-x)$. As a consequence we have that regardless of the action chosen
  by the agent, as long as the sign of the action is equal to that of
  $\ell_t$ (which happens from some point on w.p.\ 1),
  \begin{align*}
    |\ell_{t+1}| = |\ell_t| + D_+(|\ell_t|).
  \end{align*}
  Intuitively, if we can choose private signal distributions that make $D_+(x)$ decay very slowly, then $\ell_t$ will be very close to being linear.

  Formally, and by elementary considerations, the theorem will follow if, for every
  $Q \colon \R \to \R_{>0}$ with $\lim_{x\to \infty} Q(x) = 0$, we can find
  CDFs such that $D_+(x) = -D_-(-x)$ and $\liminf_{x \to \infty}D_+(x)/Q(x) > 0$.
  
  Fix any $Q$ such that $\lim_{x \to \infty}Q(x)=0$, but assume without loss of generality that $Q(x)$ is monotone
  decreasing.\footnote{If $Q$ is not monotone decreasing then consider instead $Q'(x) = \sup_{y \geq x}Q(y)$.} Define a finite measure $\nu$ on the integers by
  \begin{align*}
    \nu(n) = \frac{Q(n-1) - Q(n)}{\ee^n}
  \end{align*}
  and
  \begin{align*}
    \nu(-n) = Q(n-1) - Q(n)
  \end{align*}
  for all $n \geq 0$. Note that $\nu$ is indeed finite since
  \begin{align*}
    C := \sum_{n=-\infty}^\infty \nu(n) \leq 2 Q(-1).
  \end{align*}
  Note also that
  \begin{align*}
    \sum_{n=-\infty}^\infty \nu(n) \cdot \ee^n
  \end{align*}
  is likewise equal to $C$.
  
  Let the private signal distributions be given by
  \begin{align*}
    \mathbb{P}(s_t= n|\theta=+1) = C^{-1}\nu(n)\ee^n
  \end{align*}
  and
  \begin{align*}
    \mathbb{P}(s_t= n|\theta=-1) = C^{-1}\nu(n).
  \end{align*}
  Then 
  \begin{align*}
   L_t = \log\frac{\mathbb{P}(s_t|\theta=+1)}{\mathbb{P}(s_t|\theta=-1)}=s_t,
  \end{align*}
  the distribution of $L_t$ is identical to
  that of $s_t$, and so $G_{+}=F_{+}$ and $G_{-}=F_{-}$. By our definition of $F_{-}$, we have that for $x > 0$
  \begin{align}
  \label{eq:G}
      G_-(-x) = C^{-1}\cdot Q(\lceil x \rceil -1).
  \end{align}
  
  Now, by Lemma~\ref{lem:D}, we know that
  \begin{align*}
    (1-\epsilon)\cdot G_-(-x) < D_+(x) < (1+\epsilon) \cdot G_-(-x),
  \end{align*}
  for any $\epsilon>0$ and all $x$ large enough. It follows that
  \begin{align*}
      \liminf_{x \to \infty} \frac{D_+(x)}{Q(x)} = \liminf_{x \to \infty}\frac{G_-(-x)}{Q(x)},
  \end{align*}
  which, by~\eqref{eq:G} equals
  \begin{align*}
      \liminf_{x \to \infty}\frac{C^{-1}Q(\lceil x \rceil-1)}{Q(x)} \geq C^{-1}.
  \end{align*}

\end{proof}

\section{Long-term behavior of public belief}

The primary goal of this section is to prove Theorem~\ref{thm:diff}, which states that public belief is asymptotically given by the solution to the differential equation \eqref{eq:diff equation}. The proof of this theorem uses two general lemmas regarding recurrence relations. We state these lemmas now and prove them later. The first lemma states that two similar recurrence relations yield similar solutions. The second shows that the solution to a recurrence relation (of the type we are interested in) is well approximated by the solution to the corresponding differential equation.

\begin{lem}
\label{lem:diff_first}
Let $A,B \colon \mathbb{R}_{>0} \to \mathbb{R}_{>0}$ be continuous, eventually monotone decreasing, and tending to zero. 

Let $(a_t)$ and $(b_t)$ be sequences satisfying the recurrence relations
\begin{align*}
    a_{t+1}&=a_t+A(a_t)\\
    b_{t+1}&=b_t+B(b_t).
\end{align*}
Suppose
\begin{align*}
    \lim_{x\rightarrow\infty}\frac {A(x)}{B(x)}=1.
\end{align*}
Then
\begin{align*}
    \lim_{t\rightarrow\infty}\frac {a_t}{b_t}=1.
\end{align*}

\end{lem}
\vspace{0.25in}

\begin{lem}
\label{lem:recur_rel_soln_same_as_diff_eq_soln}

Assume that $A \colon \mathbb{R}_{>0}\rightarrow\mathbb{R}_{>0}$ is a continuous function with a convex differentiable
tail, and that $A(x)$ goes to $0$ as $x$ goes to $\infty$. Let
$(a_t)$ be any sequence satisfying the recurrence equation
$a_{t+1}=a_t+A(a_t)$, and suppose there is a function
$f:\mathbb{R}_{>0}\rightarrow\mathbb{R}_{>0}$ with $f'(t)=A(f(t))$ for all
sufficiently large $t$.  Then
\begin{align*}
  \lim_{t \rightarrow \infty}{\frac{f(t)}{a_t}} = 1.
\end{align*}
\end{lem}

Given these lemmas, we are ready to prove our theorem.
\begin{proof}[Proof of Theorem~\ref{thm:diff}]

Let $(a_t)$ be any sequence in $\mathbb{R}_{>0}$ satisfying:
\begin{align*}
    a_{t+1} = a_t + G_-(-a_t).
\end{align*}
Then by Lemma~\ref{lem:recur_rel_soln_same_as_diff_eq_soln}, the sequence $(a_t)$ is well approximated by $f(t)$, the solution to the corresponding differential equation:
\begin{align*}
    \lim_{t \rightarrow \infty}{\frac{a_t}{f(t)}} = 1.
\end{align*}
Now, conditional on $\theta = +1$, all agents take action $+1$ from some point on with probability $1$. Thus, with probability $1$,
\begin{align*}
    \ell_{t+1}=\ell_t+D_+(\ell_t)
\end{align*}
for all sufficiently large $t$. Further, by Lemma~\ref{lem:D},
\begin{align*}
    \lim_{x\rightarrow \infty}\frac {D_+(x)}{G_-(-x)}=1.
\end{align*}
So by Lemma~\ref{lem:diff_first},
\begin{align*}
    \lim_{t \rightarrow \infty}{\frac{\ell_t}{a_t}} = 1
\end{align*}
with probability $1$. Thus, we have
\begin{align*}
    \lim_{t \rightarrow \infty}{\frac{\ell_t}{f(t)}} = \lim_{t \rightarrow \infty}{\frac{\ell_t}{a_t} \cdot \frac{a_t}{f(t)}} = 1
\end{align*}
with probability $1$.
\end{proof}

\subsection{Proofs of Lemmas~\ref{lem:diff_first} and~\ref{lem:recur_rel_soln_same_as_diff_eq_soln}}

\begin{proof}[Proof of Lemma~\ref{lem:diff_first}]
We prove the claim in two steps. First, we show that for every $\varepsilon>0$ there are infinitely many times $t$ such that 
\begin{align}
    \label{eq:eps-a-b}
    (1-\varepsilon)a_t\leq b_t\leq (1+\varepsilon)a_t.
\end{align}
Second, we show that if~\eqref{eq:eps-a-b} holds for some $t$ large enough, then it holds for all $t'>t$, proving the claim. 

We start with step 1. Assume without loss of generality that
$a_t \leq b_t$ for infinitely many values of $t$. Fix
$\varepsilon>0$. To show that
$(1-\varepsilon)a_t\leq b_t\leq (1+\varepsilon)a_t$ holds for
infinitely many values of $t$, let $x_0 > 1$ be such that for all
$x > x_0$ it holds that $A$ and $B$ are monotone decreasing,
\begin{align*}
    A(x),B(x)<\varepsilon<1
\end{align*}
and
\begin{align}
    \label{eq:A-B}
    (1-\varepsilon/2)A(x) < B(x) < (1+\varepsilon/2)A(x).
\end{align}

Assume that $a_t, b_t > x_0$; this will indeed be the case for $t$
large enough, since $A$ and $B$ are positive and continuous, and so
both $a_t$ and $b_t$ are monotone increasing and tend to infinity.  So
\begin{align*}
    B(b_t) < (1+\varepsilon/2)A(b_t) \leq (1+\varepsilon/2)A(a_t),   
\end{align*}
where the first inequality follows from~\eqref{eq:A-B}, and the second follows from the fact that $A$ is monotone decreasing and $a_t < b_t$. Since $B(b(t)) = b_{t+1}-b(t)$ and $A(a_t) = a_{t+1}-a(t)$ we have shown that 
\begin{align*}
    b_{t+1} - b_t < (1+\varepsilon/2)(a_{t+1} - a_t),
\end{align*}
and so eventually $b_t \leq (1+\varepsilon)a_t$. Also, notice that the first time this obtains, we also have that the left inequality in ~\eqref{eq:eps-a-b} holds at the same moment:

$$
b_t > b_{t-1} > a_{t-1} = a_t - (a_t - a_{t-1}) > a_t - \varepsilon > a_t - \varepsilon a_t= (1-\varepsilon)a_t.
$$
This completes the first step.  Now we go to step 2. Here we show that
if~\eqref{eq:eps-a-b} holds for large enough $t$ then it holds for all
$t'>t$.

Fix $\varepsilon > 0$, and let $x_0$ be defined as above.
Suppose that $(1-\varepsilon)a_t < b_t < (1+\varepsilon)a_t$, with $a_t, b_t > x_0$. Assume without loss of generality that $b_t \geq a_t$. Then our assumptions and ~\eqref{eq:A-B} imply 
\begin{align*}
 b_{t+1} 
 &= b_t+B(b_t)\\
 &<(1+\varepsilon)a_t+(1+\varepsilon)A(b_t).
\end{align*}
Because $a_t\leq b_t$ and $A$ is decreasing we have
\begin{align*}
 b_{t+1}&<(1+\varepsilon)a_t+(1+\varepsilon)A(a_t)\\
 &=(1+\varepsilon)a_{t+1}.   
\end{align*}
For the other direction, note first that
\begin{align*}
    b_{t+1} > b_t \geq a_t,
\end{align*}
by assumption. We can write $a_t = (1-\varepsilon)a_t+\varepsilon a_t$, and since $a_t > x_0 > 1$, $\varepsilon a_t > (1-\varepsilon)\varepsilon$, and so
\begin{align*}
    b_{t+1} >(1-\varepsilon)a_t+(1-\varepsilon)\varepsilon.
\end{align*}
Now, $\varepsilon > A(a_t)$ since $a_t > x_0$, and so
\begin{align*}
    b_{t+1}
    &>(1-\varepsilon)a_t+(1-\varepsilon)A(a_t)\\
    &=(1-\varepsilon)a_{t+1}.
\end{align*}
Thus
\begin{align}
    (1-\varepsilon)a_{t+1} < b_{t+1} < (1+\varepsilon)a_{t+1},
\end{align}
as required.

\end{proof}

\begin{proof}[Proof of Lemma~\ref{lem:recur_rel_soln_same_as_diff_eq_soln}]

We restrict the domain of $f$ to the interval $(t_0,\infty)$ such that for $t>t_0$ it already holds that $f'(t) = A(f(t))$. Since $A$ is continuous, $\lim_{t \to \infty} f(t)=\infty$, and so we can also assume that in the  interval $(f(t_0),\infty)$ it holds that $A$ is convex and differentiable. 

Since $f$ is strictly increasing in $(t_0,\infty)$, it has an inverse $f^{-1}$. For $x$ large enough define $B(x)=f(f^{-1}(x)+1)-x$. 

Now, let $(b_t)$ be any sequence satisfying the recurrence relation
\begin{align*}
    b_{t+1} = b_t + B(b_t).
\end{align*}
In order to apply Lemma~\ref{lem:diff_first}, we will first show that
\begin{align*}
    \lim_{x\rightarrow\infty}\frac{B(x)}{A(x)}=1.
\end{align*}

Let $t = f^{-1}(x)$. Such a $t$ exists and is unique for all
sufficiently large $x$, because $f$ is monotone. Notice that by the
definitions of $B(x)$ and $f'(x)$
\begin{align*}
    B(x)&=f(f^{-1}(x)+1)-x\\
    &=f(f^{-1}(x)+1)-x-f'(f^{-1}(x))+f'(f^{-1}(x))\\
    &=f(t+1)-f(t)-f'(t)+A(f(t)),
\end{align*}
where in the last equality we substitute $t = f^{-1}(x)$.
Because $f'$ is positive and decreasing ($f$ is concave) then  $f(t+1)-f(t)\geq f'(t+1) $, and so 
\begin{align*}
    B(x) &\geq f'(t+1)-f'(t)+A(f(t)).
\end{align*}
By the definition of $f$, $f'(t) = A(f(t))$, and so
\begin{align*}
    B(x) \geq A(f(t+1))-A(f(t))+A(f(t)) =A(f(t+1)).
\end{align*}
Again, due to concavity of $f$ we have $f(t+1)\leq f(t)+f'(t)$ and as $A$ is decreasing and convex we get
\begin{align*}
    B(x) &\geq A(f(t)+f'(t))\\
    &\geq A'(f(t))f'(t)+A(f(t))\\
    &=A'(f(t))A(f(t))+A(f(t)).
\end{align*}
We now substitute back $x=f(t)$:
\begin{align*}
   B(x)&\geq A'(x)A(x) + A(x)\\
    &=A(x)(A'(x) + 1)
\end{align*}
so in particular, since $A'(x) \rightarrow 0$ as $x \rightarrow \infty$,
\begin{align*}
    \liminf_{x \rightarrow \infty}{\frac{B(x)}{A(x)}} \geq 1.
\end{align*}

Now we are going to show that $\limsup_{x\rightarrow \infty}\frac{B(x)}{A(x)}\leq 1$ which will conclude the proof. By the definitions of $f^{-1}(x)$ and $B(x)$ 
\begin{align*}
    B(x) = B(f(t)) = f(t+1) - f(t) = \int_t^{t+1}f'(\zeta)\,\dd \zeta.
\end{align*}
As $f'$ is decreasing it follows that 
\begin{align*}
    B(x)&\leq \int_t^{t+1}f'(t)\,\dd \zeta = f'(t) = A(f(t)) = A(x).
\end{align*}
Therefore, 
\begin{align*}
    \limsup_{x\rightarrow\infty}{\frac{B(x)}{A(x)}}\leq 1.
\end{align*}
Hence, from these two inequalities we get that 
\begin{align*}
    \lim_{x \rightarrow \infty}{\frac{B(x)}{A(x)}=1}.
\end{align*}
Now notice that, by construction, $f(t+1)=f(t)+B(f(t))$.
Thus, by Lemma~\ref{lem:diff_first},
\begin{align*}
    \lim_{n \rightarrow \infty}{\frac{f(t)}{a_t}=1}.
\end{align*}

\end{proof}

\subsection{Monotonicity of solutions to a differential equation}

We now prove a general lemma regarding differential equations of the
form $a'(t) = A(a(t))$. It shows that the solutions to this equation
are monotone in $A$. This is useful for calculating approximate
analytic solutions whenever it is impossible to find analytic exact
solutions, as is the case of Gaussian signals, in which we use this
lemma.
\begin{lem}
\label{lem:ineq_of_aut_dif_eqs}
Let $A, B \colon \mathbb{R}_{>0} \to \mathbb{R}_{>0}$ be continuous, and let $a, b \colon \mathbb{R}_{>0} \to \mathbb{R}_{>0}$ satisfy $a'(t) = A(a(t))$ and $b'(t) = B(b(t))$ for all sufficiently large $t$. 

Suppose that
\begin{align*}
    \liminf_{x \rightarrow \infty}{\frac{A(x)}{B(x)}} > 1.
\end{align*}
Then $a(t) > b(t)$ for all sufficiently large $t$.
\end{lem}

\begin{proof}
Notice that $a(t)$ and $b(t)$ are eventually monotone increasing and tend to infinity as $t$ tends to infinity. Thus for all $x$ greater than some $x_0 > 0$ large enough, $a$ and $b$ have inverses that satisfy the following differential equations:
\begin{align*}
    \frac{\dd}{\dd x}a^{-1}(x) &= \frac{1}{A(x)}\\
    \frac{\dd}{\dd x}b^{-1}(x) &= \frac{1}{B(x)}\cdot
\end{align*}
Since $\liminf_x A(x)/B(x)>1$, we can furthermore choose $x_0$ so that for all $x \geq x_0$, $A(x) > (1 + \varepsilon) B(x)$ for some $\varepsilon > 0$. Thus, for $x > x_0$
\begin{align*}
    a^{-1}(x) &= a^{-1}(x_0) + \int_{x_0}^{x}{\frac{1}{A(x)}\, \dd x}\\
    b^{-1}(x) &= b^{-1}(x_0) + \int_{x_0}^{x}{\frac{1}{B(x)}\, \dd x}
\end{align*}
and so
\begin{align*}
    a^{-1}(x) &< a^{-1}(x_0) +
                \frac{1}{1+\varepsilon}\int_{x_0}^{x}{\frac{1}{B(x)}\,\dd
                x}\\
    &= a^{-1}(x_0) + \frac{1}{1+\varepsilon}(b^{-1}(x) - b^{-1}(x_0))
\end{align*}
and thus
\begin{align*}
    a^{-1}(x) - b^{-1}(x) < -\frac{\varepsilon}{1+\varepsilon}b^{-1}(x) + \left[ a^{-1}(x_0) - \frac{1}{1+\varepsilon}b^{-1}(x_0) \right].
\end{align*}
Since $b^{-1}(x)$ tends to infinity as $x$ tends to infinity, it follows that for all sufficiently large $x$, $a^{-1}(x) < b^{-1}(x)$. Thus, for all sufficiently large $t$
\begin{align*}
    t = a^{-1}(a(t)) < b^{-1}(a(t)),
\end{align*}
and so, since $b(t)$ is monotone increasing,
\begin{align*}
    b(t) < a(t).
\end{align*}
\end{proof}

\subsection{Eventual monotonicity of public belief update}

We end this section with a lemma that shows that under some technical
conditions on the left tail of $G_-$, the function
$u_+(x) = x + D_+(x)$ (i.e., the function that determines how the
public log-likelihood ratio is updated when the action $+1$ is taken)
is eventually monotone increasing.
\begin{lem}
\label{lem:u_eventually_monotonic}
Suppose $G_-$ has a convex and differentiable left tail. Then the map $u_+(x) = x + D_+(x)$ is monotone increasing for all sufficiently large $x$.
\end{lem}
\begin{proof}
Recall that 
\begin{align*}
    D_+(x)=\log\frac{1-G_+(-x)}{1-G_-(-x)}\cdot
\end{align*}
Since $G_-$ has a differentiable left tail, it has a derivative $g_-(-x)$ for all $x$ large enough. It then follows from \eqref{eq:nu} that $G_+$ also has a derivative in this domain, and
\begin{align*}
    u_+'(x)
    &=1+\frac{g_+(-x)}{1-G_+(-x)}-\frac{g_-(-x)}{1-G_-(-x)}\\
    &=1+\frac{\ee^{-x} g_-(-x)}{1-G_+(-x)}-\frac{g_-(-x)}{1-G_-(-x)}\cdot
\end{align*}
Since $1-G_-(-x)$ and $1-G_+(-x)$ tend to $1$ as $x$ tends to infinity,
\begin{align*}
    \lim_{x \to \infty}u_+'(x) = \lim_{x \to \infty} 1 + \ee^{-x}g_-(-x)-g_-(-x).
\end{align*}
Since $G_-$ is eventually convex, $g_-(-x)$ tends to zero, and therefore
\begin{align*}
    \lim_{x \to \infty}u_+'(x) = 1.
\end{align*}
In particular, $u_+'(x)$ is positive for $x$ large enough, and hence $u_+(x)$ is eventually monotone increasing.
\end{proof}

\section{Gaussian private signals}
\label{app:proofs}
\subsection{Preliminaries}
We say that private signals are Gaussian when $F_-$ is the normal distribution with mean $-1$ and variance $\sigma^2$, and $F_+$ is the normal distribution with mean $+1$ and variance $\sigma^2$. To calculate the evolution of $\ell_t$, we need to calculate $G_+$ and $G_-$, the conditional distributions of the private log-likelihood ratio $L_t$. Notice that in this case
\begin{align*}
  L_t = \log{\frac{\ee^{-(s_t - 1)^2/2\sigma^2}}{\ee^{-(s_t -
  (-1))^2/2\sigma^2}}} = 2s_t/\sigma^2,
\end{align*}
so that $L_t$ is simply proportional to the signal $s_t$. 
It follows that $L_t$ is also normally distributed, conditioned on the state
$\theta$, and that $G_+$ and $G_-$ are cumulative distribution functions of Gaussians, with variance  $4/\sigma^2$.

\subsubsection{Notation}
In this section and those that follow, we denote by $\ell_t^*$ the public log-likelihood ratio when all agents before agent $t$ take the correct action. Formally,
\begin{align*}
    \ell_t^* = \log{\frac{\mathbb{P}(\theta = +1\, |\, a_1 = \cdots = a_{t-1} = +1)}{\mathbb{P}(\theta = -1\, |\, a_1 = \cdots = a_{t-1} = +1)}}\cdot
\end{align*}
For convenience, we will also use the notation $\mathbb{P}_+(\cdot)$ as shorthand for $\mathbb{P}(\cdot \, |\, \theta = +1)$.

\subsection{The evolution of public belief}

\begin{proof}[Proof of Theorem~\ref{l_t_asymptotics}]

Let $f \colon \mathbb{R}_{>0} \rightarrow \mathbb{R}_{>0}$ be any function such that $f'(t) = G_{-}(-f(t))$ for all sufficiently large $t$. Then by Theorem~\ref{thm:diff},
\begin{align*}
    \lim_{t \rightarrow \infty}{\frac{\ell_t}{f(t)}} = 1
\end{align*}
with probability $1$.

Recall from above that $L_t$ is distributed normally, and $G_{-}(-x)$ is the CDF of a normal distribution with variance $\tau^2=4/\sigma^2$.

For $1 > \eta \geq 0$, define
\begin{align*}
    F_{\eta}(x) &= \frac{\ee^{-\frac{1-\eta}{2\tau^2}x^2}}{x}\\
    f_{\eta}(t) &= \frac{\sqrt{2}\tau}{\sqrt{1-\eta}}\sqrt{\log(t)+\log\frac{(1-\eta)^2}{2\tau^2}}\cdot
\end{align*}
By a routine application of L'Hospital's rule, $F_0$ and $F_\eta$ are lower and upper bounds for $G_-$, in the sense that
\begin{align*}
    \lim_{x \rightarrow \infty}{\frac{G_{-}(-x)}{F_{0}(x)}} &= \infty\\
    \lim_{x \rightarrow \infty}{\frac{F_{\eta}(x)}{G_{-}(-x)}} &= \infty \text{, } \eta > 0.
\end{align*}
Since $f_{\eta}'(t) = F_{\eta}(f_{\eta}(t))$ for all sufficiently large $t$, we have by Lemma~\ref{lem:ineq_of_aut_dif_eqs} that for any $\eta > 0$,
\begin{align*}
      f_0(t) < f(t) < f_{\eta}(t)
\end{align*}
for all sufficiently large $t$. So
\begin{align*}
    \liminf_{t \rightarrow \infty}{\frac{f(t)}{\sqrt{2} \tau \sqrt{\log{t}}}} = \liminf_{t \rightarrow \infty}{\frac{f(t)}{f_0(t)}} \geq 1
\end{align*}
and for any $\eta > 0$,
\begin{align*}
    \limsup_{t \rightarrow \infty}{\frac{f(t)}{\sqrt{2} \tau \sqrt{\log{t}}}} = \frac{1}{\sqrt{1 - \eta}} \cdot \limsup_{t \rightarrow \infty}{\frac{f(t)}{f_{\eta}(t)}} \leq \frac{1}{\sqrt{1 - \eta}}\cdot
\end{align*}
Thus,
\begin{align*}
    \lim_{t \rightarrow \infty}{\frac{f(t)}{\sqrt{2} \tau \sqrt{\log{t}}}} = \lim_{t \rightarrow \infty}{\frac{f(t)}{(2\sqrt{2}/\sigma ) \sqrt{\log{t}}}} 
    = 1
\end{align*}
so with probability $1$,
\begin{align*}
    \lim_{t \rightarrow \infty}{\frac{\ell_t}{(2\sqrt{2}/\sigma)  \sqrt{\log{t}}}} = \lim_{t \rightarrow \infty}{\frac{\ell_t}{f(t)} \cdot \frac{f(t)}{(2\sqrt{2}/\sigma)  \sqrt{\log{t}}}} = 1.
\end{align*}

\end{proof}

To prove Theorem~\ref{probability_of_mistake}, we will need two lemmas. The first is general, and will be used several times in the sequel, while the second deals exclusively with the Gaussian case. 

Denote by $E_t$ the event that $a_\tau=+1$ for all $\tau \geq t$; that is, that there are no more mistakes after time $t$. The next lemma provides a uniform bound for the probability of $E_t$, conditioned on the public belief. It implies, in particular, that the probability of $E_1$ is positive, which we will use in the proof of Theorem~\ref{probability_of_mistake}.
\begin{lem}
\label{lem:all-correct}
Suppose $G_-$ and $G_+$ are continuous, and $G_{-}$ has a convex and differentiable left tail. Then for every $L \in \mathbb{R}$, there is some $m_L > 0$ such that for any $t$, $x \geq L$ implies $\mathbb{P}_{+}(E_t \, | \, \ell_t = x) \geq m_L$.
\end{lem}
\begin{proof}

Recall the definition of the public belief $\mu_t = \mathbb{P}(\theta = +1|a_1,\ldots,a_{t-1})$. The process $(\mu_1,\mu_2,\ldots)$ is a bounded martingale, and therefore, by a standard argument on bounded martingales, if we condition on $\mu_t = q$, then the probability that $\mu_\tau \leq 1/2$ for some $\tau > t$ is at most $2(1-q)$.\footnote{Intuitively, if I assign high belief now to the event $\theta=+1$, then the probability that I assign this event low belief in the future must be small.} This event is precisely the complement of $E_t$, and therefore we have that $\mathbb{P}(E_t\,|\,\mu_t = q)$ is at least $2q-1$. Hence, conditioning on $\theta=+1$, we have that $\mathbb{P}_+(E_t\,|\,\mu_t = 1-q) \geq (2q-1)/q$, which is positive for all $q > 1/2$. 

Since $\mu_t=q$ is equivalent to $\ell_t = \log q/(1-q)$, what we have shown implies that there is an $\varepsilon>0$ such that for all $x \geq 1$ (here the choice of 1 is arbitrary and can be replaced with any positive number) 
\begin{align*}
    \mathbb{P}_+(E_t\,|\, \ell_t=x) > \varepsilon.
\end{align*}

Now, for any $L < 1$, the compactness of the interval $[L,1]$, together with the continuity of $G_-$ and $G_+$, implies that  there is an $n_L$ such that if $\ell_t \geq L$, and if agents $t$ through $t+n_L-1$ take action $+1$, then $\ell_{t+n_L} > 1$. Further, since the probability of agents $t$ through $t + n_L - 1$ all taking action $+1$ conditional on $\ell_t = x$ is continuous in $x$, there is a $p_L > 0$ such that
\begin{align*}
    \mathbb{P}_+(E_t\,|\,\ell_t = x) \geq p_L \cdot \varepsilon
\end{align*}
since with probability at least $p_L$ there are no mistakes up to time $t+n_L$, and thence there are no mistakes with probability at least $\varepsilon$.

\end{proof}

\begin{lem}
\label{lem:mistake_given_no_previous_mistakes}
Assume private signals are Gaussian. For every $\varepsilon > 0$ there exists some $k > 0$ such that for all $t$,
\begin{align*}
    \mathbb{P}_+(a_t = -1\, |\, a_\tau = +1 \text{ for all } \tau < t) > \frac{k}{t^{1 + \varepsilon}}\cdot
\end{align*}
\end{lem}

\begin{proof}
By the definitions of $\ell_t^*$ and $G_{+}$,
\begin{align*}
    \mathbb{P}_{+}(a_t = -1 \, | \, a_\tau = +1 \text{ for all } \tau < t) &= \mathbb{P}_{+}(a_t = -1 \,|\, \ell_t = \ell_t^*)\\
    &= G_{+}(-\ell_t^*).
\end{align*}
Now, by Theorem~\ref{l_t_asymptotics}, for every $\beta > 0$, $\ell_t^* < (1+\beta)\frac{2\sqrt{2}}{\sigma}\sqrt{\log{t}}$ for all sufficiently large $t$. Further, it follows from a routine application of L'Hopital's rule (or from the standard asymptotic expansion for the CDF of a normal distribution) that for all sufficiently large $x$,
\begin{align*}
    G_{+}(-x) > \frac{\ee^{-(\sigma^2/8)x^2}}{x}\cdot
\end{align*}
Let $\varepsilon > 0$, and take $\beta < \sqrt{1 + \varepsilon} - 1$. Then by monotonicity of $G_{+}(-x)$ and a straightforward calculation,
\begin{align*}
    G_{+}(-\ell_t^*) &> G_{+}(-(1+\beta)\frac{2\sqrt{2}}{\sigma}\sqrt{\log{t}})\\
    &> \left[ \frac{1}{(1+\beta)\frac{2\sqrt{2}}{\sigma}} \right] \cdot \frac{t^{(1 + \varepsilon) - (1 + \beta)^2}}{\sqrt{\log{t}}} \cdot \frac{1}{t^{1 + \varepsilon}}\\
    &> \frac{1}{t^{1 + \varepsilon}}
\end{align*}
for all sufficiently large $t$. From this, the claim follows immediately.

\end{proof}

\begin{proof}[Proof of Theorem~\ref{probability_of_mistake}]
Denote by $C_t$ be the event that $a_\tau = +1$ for all $\tau < t$, and note that the event $T_1=t$ is simply the intersection of $C_t$ with the event that $a_t=-1$. 

Let $\varepsilon > 0$. 
By Lemma~\ref{lem:mistake_given_no_previous_mistakes} there is some $k' > 0$ such that for all $t$,
\begin{align*}
    \mathbb{P}_+(a_t = -1\, |\, C_t) > \frac{k'}{t^{1 + \varepsilon}}\cdot
\end{align*}
Now, put $\gamma = \mathbb{P}_+(a_\tau = +1\text{ for all } \tau \geq 1)$, the probability that all agents take the correct action. By Lemma~\ref{lem:all-correct}, $\gamma > 0$, so this provides a lower bound on the probability of the first $t-1$ agents taking the correct action. Formally,
\begin{align*}
    \mathbb{P}_+(C_t) \geq \mathbb{P}_+(a_\tau = +1\, \text{ for all } \tau \geq 1) = \gamma.
\end{align*}
Thus,
\begin{align*}
    \mathbb{P}_+(T_1 = t) &= \mathbb{P}_+(a_t = -1,C_t)\\
    &= \mathbb{P}_+(a_t = -1\, |\, C_t) \cdot \mathbb{P}_+(C_t)\\
    &\geq \frac{\gamma k'}{t^{1 + \varepsilon}}
\end{align*}
for all $t$.

\end{proof}

\section{Upsets and runs}

We recall a few definitions from Section~\ref{subsec:expected_time_to_learn}. 
We say that there is an {\em upset}  at time $t$ if $a_{t-1} \neq a_t$. We denote by $\Xi$ the random variable which assigns to each outcome the total number of upsets, and by $\Xi_t$ the total number of upsets at times up to and including $t$.
We say that there is a {\em run} of length $m$ from $t$ if $a_t = a_{t+1} = \cdots = a_{t+m-1}$. Note that this definition does not preclude a run from being part of a longer run; we will refer to a run of finite length which is not strictly contained in any other run as {\em maximal}. We say that a run from $t$ is {\em good} if $a_t = +1$ and {\em bad} otherwise.

Notice that the number of maximal runs is exactly equal to the number of upsets. We use this observation now to show that the probability of having many maximal runs is very small, so that most of the probability is concentrated in the outcomes with few maximal runs.

\begin{proof}[Proof of Proposition~\ref{prop:exponential_decay_of_runs}]

Denote by $\Upsilon$ the random variable which assigns to each outcome the number of finite maximal good runs it contains; note that with probability $1$, $\Upsilon$ is finite.

By Lemma~\ref{lem:all-correct}, there is a $\beta > 0$ such that for any $x \geq 0$, if $\ell_t = x$, then the probability that all agents from $t$ on take the correct action is at least $\beta$. Formally,\footnote{We remind the reader that $\mathbb{P}_+(\cdot)$ is shorthand for $\mathbb{P}(\cdot\,|\,\theta = +1)$.}
\begin{align*}
    \mathbb{P}_+(a_\tau = +1\text{ for all } \tau \geq t\, |\, \ell_t = x) \geq \beta.
\end{align*}
Thus, whenever $a_{t-1} = -1$ and $a_t = +1$ (or $t = 1$), the probability that there is exactly one more maximal good run is at most $1 - \beta$. It follows that for $n \geq 0$,
\begin{align*}
    \mathbb{P}_+(\Upsilon = n + 1) \leq (1-\beta) \mathbb{P}_+(\Upsilon = n)
\end{align*}
and thus, for any $n \geq 0$,
\begin{align*}
    \mathbb{P}_+(\Upsilon = n) \leq (1-\beta)^n \mathbb{P}_+(\Upsilon = 0)
\end{align*}
and so
\begin{align*}
    \mathbb{P}_+(\Upsilon \geq n) &\leq \frac{\mathbb{P}_+(\Upsilon = 0)}{\beta} \cdot (1-\beta)^{n}.
\end{align*}
Finally, since $\Upsilon = \lfloor \Xi / 2 \rfloor$, we have for any $n$:
\begin{align*}
    \mathbb{P}_+(\Xi \geq n) &\leq \mathbb{P}_+(\Upsilon \geq \lfloor n / 2 \rfloor) \leq c \cdot \gamma^n
\end{align*}
where $c = \mathbb{P}_+(\Upsilon = 0)/\beta$ and $\gamma = (1-\beta)^{\frac{1}{3}}$.

\end{proof}

Whenever asymptotic learning occurs (that is, whenever the probability that all agents take the correct action from some point on is equal to $1$), the total number of upsets is almost surely finite. In particular, the probability that $\Xi_t$ is logarithmic in $t$ tends to zero as $t$ tends to infinity.
Using Proposition~\ref{prop:exponential_decay_of_runs}, we can show that in fact this probability tends to $0$ quickly:
\begin{corollary}
\label{corollary:upsets exponential}
Let $c, \gamma$ be as in Proposition~\ref{prop:exponential_decay_of_runs}. Then
\begin{align*}
    \mathbb{P}(\Xi_t \geq -\frac{2.1}{\log{\gamma}}\log{t}) \leq c \cdot \frac{1}{t^{2.1}}\cdot
\end{align*}
\end{corollary}

\begin{proof}
\begin{align*}
    \mathbb{P}(\Xi_t \geq -\frac{2.1}{\log{\gamma}}\log{t})&\leq \mathbb{P}(\Xi \geq -\frac{2.1}{\log{\gamma}}\log{t})\\
    &\leq c \cdot\gamma^{-\frac{2.1}{\log{\gamma}}\log{t}}\\
    &=c \cdot \frac{1}{t^{2.1}}\cdot
\end{align*}
\end{proof}

In fact, it is equally easy to show the same statement for exponents larger than $2.1$, but this will suffice for our purposes.

One important consequence of Corollary~\ref{corollary:upsets exponential} is that with high probability, there is at least one maximal run before time $t$ which is long relative to $t$. Thus, much of the dynamics is controlled by what happens during long runs.

We previously analyzed only long runs that start at time $1$, when the public log-likelihood ratio is equal to $0$. If a long run starts at some  public belief $\ell_t\neq 0$ then its evolution is different from the former case. However, if the run is long enough then the analysis above can still be applied. The following lemma states that if a run starts at some $\ell_t>0$ then we can bound the future public belief from below using $\ell^*$.

\begin{lem}
\label{lem:quickly_same_as_l_star}
Suppose that $G_{-}$ has a convex and differentiable left tail. Then there exists a $z > 0$ such that, if there is a good run of length $s$ from $t$, then $\ell_{t+s} \geq \ell_{s-z}^*$.
\end{lem}

\begin{proof}

Let $u_{+}(x) = x + D_+(x)$. Then by \eqref{eq:ell_t}, whenever agent $t$ takes action $+1$, $\ell_{t+1} = u_{+}(\ell_t)$.

Since $G_{-}$ is eventually convex and differentiable, $u_{+}(x)$ is monotone increasing for sufficiently large $x$, by Lemma~\ref{lem:u_eventually_monotonic}. Take 
$$z = \min{\{t \in \mathbb{N} \colon u_{+}(x) \text{ is monotone on } (\ell_{t}^*-1, \infty)\}}.$$ 
Now, let $\mu = \inf_{x \in [0, \ell_{z}^*]}{D_+(x)}$. By continuity of $D_+(x)$ and compactness of $[0, \ell_{z}^*]$, $\mu > 0$, since $D_{+}(x) > 0$ for all $x$. Put $N = \lceil \frac{\ell_{z}^*}{\mu}\rceil$. Then for all $x \in [0, \ell_{z}^*]$, $u_{+}^{N}(x) \geq \mu \cdot N \geq \ell_{z}^*$. Further, since $u_{+}(x) > x$ for all $x$, it follows that whenever there is a run of length $N$ from $t$, $\ell_{t+N} > \ell_{z}^*$.

This implies that if there is a good run from $t$ of length $s\geq N$, then $\ell_{t+s}\geq \ell_{s-z}^*$.

\end{proof}

\section{Distributions with polynomial tails}

In this appendix we prove Theorem~\ref{thm:polytail-finite-expectation}, showing that for private log-likelihood distributions with polynomial tails, the expected time to learn is finite.

As in the setting of Theorem~\ref{thm:polytail-finite-expectation},
assume that the conditional distributions of the private
log-likelihood ratio satisfy
\begin{align}
\label{eq:poly_G_def_plus}
    G_+(x) &= 1 - \frac{c}{x^k} \text{ for all } x > x_0\\
\label{eq:poly_G_def_minus}
    G_-(x) &= \frac{c}{(-x)^{k}} \text{ for all } x < -x_0
\end{align}
for some $x_0 > 0$.

We remind the reader that we denote by $\ell_t^*$ the log-likelihood ratio of the public belief that results when the first $t-1$ agents take action $+1$. It follows from Theorem~\ref{thm:diff} that in this setting, $\ell_t^*$ behaves asymptotically as $t^{1/(k+1)}$. Notice also that, by the symmetry of the model, the log-likelihood ratio of the public belief that results when the first $t-1$ agents take action $-1$ is $-\ell_t^*$.

We begin with the simple observation that a strong enough bound on the probability of mistake is sufficient to show that the expected time to learn is finite. Formally, we have the following lemma. We remind the reader that $\mathbb{P}_+(\cdot)$ is shorthand for $\mathbb{P}(\cdot\,|\,\theta = +1)$.
\begin{lem}
\label{lem:mistake-bounds-time-to-learn}
Suppose there exist $k$, $\varepsilon > 0$ such that for all $t \geq 1$, $\mathbb{P}_{+}(a_t = -1) < k \cdot \frac{1}{t^{2+\varepsilon}}$. Then $\mathbb{E}_+(T_L)$ is finite.
\end{lem}
\begin{proof}
Since $T_L = t$ only if $a_{t-1} = -1$, $\mathbb{P}_{+}(T_L = t) \leq \mathbb{P}_{+}(a_{t-1} = -1)$. Thus
\begin{align*}
    \mathbb{E}_+(T_L) &= \sum_{t=1}^{\infty} t\cdot\mathbb{P}_{+}(T_L = t)\\
    &\leq \mathbb{P}_+(T_L = 1) + \sum_{t=2}^{\infty} t\cdot\mathbb{P}_{+}(a_{t-1}=-1)\\
    &\leq 1 + k \sum_{i=2}^{\infty} \frac{t}{(t-1)^{2 + \varepsilon}}\\
    &< \infty.
\end{align*}
\end{proof}
Accordingly, this section will be primarily devoted to studying the
rate of decay of the probability of mistake,
$\mathbb{P}_+(a_{t} = -1)$.  In order to bound this probability, we
will need to make use of the following lemmas, which give some control
over how the public belief is updated following an upset.

\begin{lem}
\label{lem:dont go below too far}
For $G_+$ and $G_-$ as in \eqref{eq:poly_G_def_plus} and \eqref{eq:poly_G_def_minus}, $|\ell_{t+1}|\leq|\ell_t|$ whenever $|\ell_t|$ is sufficiently large and $a_t \neq a_{t+1}$.
\end{lem}

\begin{proof}

Assume without loss of generality that $a_t = +1$ and $a_{t+1}=-1$, so that
\begin{align*}
    \ell_{t+1}=\ell_t+D_-(\ell_t).
\end{align*}
Thus, to prove the claim we compute a bound for $D_-$. To do so we first obtain a bound for the left tail of $G_+$. By assumption, for $x > x_0$ (with $x_0$ as in \eqref{eq:poly_G_def_plus} and \eqref{eq:poly_G_def_minus}),
\begin{align*}
    g_-(-x) = G_-'(-x) = \frac{ck}{x^{k+1}}
\end{align*}
and so by \eqref{eq:nu},
\begin{align*}
    g_+(-x) = \ee^{-x} g_-(-x) = ck\frac{\ee^{-x}}{x^{k+1}}.
\end{align*}
Hence, 
\begin{align*}
    G_+(-x) = \int_{-\infty}^{-x}{g_+(\zeta) \,\dd\zeta} 
    = \int_{-\infty}^{-x}{ck\frac{\ee^\zeta}{(-\zeta)^{k+1}} \,\dd\zeta}
    = ck \int_{x}^{\infty}{\zeta^{-k-1}\ee^{-\zeta}\,\dd\zeta}.
\end{align*}
For $\zeta$ sufficiently large, $\zeta^{-k-1}$ is at least, say, $\ee^{-.1\zeta}$. Thus, for $x$ sufficiently large,
\begin{align*}
    G_+(-x) \geq ck\int_{x}^{\infty}{\ee^{-1.1\zeta}\,\dd\zeta} = \frac{ck}{1.1} \ee^{-1.1x}.
\end{align*}
It follows that for $x$ sufficiently large,
\begin{align*}
    D_{-}(x) = \log{\frac{G_+(-x)}{G_-(-x)}} \geq \log{\frac{ck}{1.1}} - 1.1x + k\log{x} \geq -1.2x.
\end{align*}
Thus, for $\ell_t$ sufficiently large,
\begin{align*}
    \ell_{t+1} = \ell_t + D_-(\ell_t) = \ell_t + \log\frac{G_+(-\ell_t)}{G_-(-\ell_t)} \geq \ell_t + 1.2(-\ell_t) = -.2\ell_t
\end{align*}
so in particular, $|\ell_{t+1}| < |\ell_t|$.

\end{proof}


We will make use of the following lemma, which bounds the range of possible values that $\ell_t$ can take.

\begin{lem}
\label{lem:ratio to ell star bounded}
For $G_{+}$ and $G_{-}$ as in \eqref{eq:poly_G_def_plus} and \eqref{eq:poly_G_def_minus}, there exists an $M > 0$ such that for all $t \geq 0$, $|\ell_s| \leq M \cdot \ell_{t}^*$ for all $s \leq t$.
\end{lem}

\begin{proof}

For each $\tau \geq 0$, define 
\begin{align*}
    M_\tau = \max{\frac{|\ell_\tau|}{\ell_\tau^*}}
\end{align*}
where the maximum is taken over all outcomes. Note that there are at most $2^\tau$ possible values for this expression, so $M_\tau$ is well-defined and finite. Put
\begin{align*}
    M = \sup_{\tau \geq 0}{M_\tau}.
\end{align*}
To establish the claim, we must show that $M$ is finite. To do this, it suffices to show that for $\tau$ sufficiently large, $M_{\tau + 1} \leq M_\tau$.

Now, let $u_{+}(x) = x + D_+(x)$ and $u_{-}(x) = x + D_-(x)$. Then as shown in the section about the model, whenever agent $\tau$ takes action $+1$, $\ell_{\tau+1} = u_{+}(\ell_\tau)$, and whenever agent $\tau$ takes action $-1$, $\ell_{\tau+1} = u_{-}(\ell_\tau)$.

By Lemma~\ref{lem:u_eventually_monotonic}, $u_+$ and $u_-$ are eventually monotonic. Thus, there exists $x_0 > 0$ such that $u_{+}$ is monotone increasing on $(x_0, \infty)$ and $u_{-}$ is monotone decreasing on $(-\infty, -x_0)$. 

For $\tau$ sufficiently large, $\ell_{\tau}^* > x_0$. Further, it follows from Lemma~\ref{lem:dont go below too far} that for $\tau$ sufficiently large, $|\ell_{\tau+1}| < |\ell_{\tau}|$ whenever $a_\tau \neq a_{\tau+1}$ and $|\ell_\tau| > |\ell_\tau^*|$. Let $(a_\tau)$ be any sequence of actions with $\frac{|\ell_{\tau + 1}|}{\ell_{\tau + 1}^*} = M_{\tau + 1}$. If $a_\tau \neq a_{\tau + 1}$ 
\begin{align*}
    M_{\tau + 1} &= \frac{|\ell_{\tau+1}|}{\ell_{\tau + 1}^*}
    \leq \frac{|\ell_{\tau}|}{\ell_{\tau}^*}
    \leq M_\tau.
\end{align*}
If $a_\tau = a_{\tau + 1}$, then either $M_{\tau + 1} = 1$, in which case $M_{\tau + 1} \leq M_{\tau}$, or $M_{\tau + 1} > 1$. If $M_{\tau + 1} > 1$, then since $|D_{+}|$ and $|D_{-}|$ are decreasing on $(x_0, \infty)$ and $(-\infty, -x_0)$ respectively, $|\ell_{\tau+1} - \ell_{\tau}|/|\ell_{\tau}| \leq |\ell_{\tau+1}^* - \ell_{\tau}^*|/|\ell_{\tau}^*|$. So
\begin{align*}
    M_{\tau + 1} &= \frac{|\ell_{\tau+1}|}{\ell_{\tau + 1}^*} = \frac{|\ell_{\tau}| + |\ell_{\tau+1} - \ell_{\tau}|}{\ell_{\tau}^* + |\ell_{\tau+1}^* - \ell_{\tau}^*|}
\end{align*}
where the second equality follows from the fact that $\ell_{\tau}$ and $\ell_{\tau + 1}$ have the same sign. Finally, 
\begin{align*}
    M_{\tau + 1} &= \frac{|\ell_{\tau}|}{\ell_{\tau}^*} \cdot \frac{1 + |\ell_{\tau+1} - \ell_{\tau}|/|\ell_{\tau}|}{1 + |\ell_{\tau+1}^* - \ell_{\tau}^*|/\ell_{\tau}^*}
    \leq \frac{|\ell_{\tau}|}{\ell_{\tau}^*}
    \leq M_{\tau}.
\end{align*}

Thus, for all sufficiently large $\tau$, $M_{\tau + 1} \leq M_\tau$.

\end{proof}

\begin{proposition}
\label{proposition:polynom_probability}
There exists $\kappa > 0$ such that $\mathbb{P}_{+}(a_t = -1) < \kappa t^{-2.1}$ for all $t > 0$.
\end{proposition}

\begin{proof}
Let $\beta = -2.1/\log{\gamma}$, where $\gamma$ is as in Proposition~\ref{prop:exponential_decay_of_runs}. To carry out our analysis, we will divide the event that $a_t = -1$ into three disjoint events and bound each of them separately:
\begin{align*}
    A &= (a_t = -1) \text{ and } (\Xi_t > \beta\log{t})\\
    B_1 &= (a_t = -1) \text{ and } (\Xi_t \leq \beta\log{t}) \text{ and } (|\{s \colon s < t, a_s = +1\}| \geq \frac{1}{2}t)\\
    B_2 &= (a_t = -1) \text{ and } (\Xi_t \leq \beta\log{t}) \text{ and } (|\{s \colon s < t, a_s = +1\}| < \frac{1}{2}t).
\end{align*}

First, by Corollary~\ref{corollary:upsets exponential} we have a bound for $\mathbb{P}_+(A)$
\begin{align*}
    \mathbb{P}_+(A)\leq c\cdot\frac{1}{t^{2.1}}.
\end{align*}

Next, we bound $\mathbb{P}_{+}(B_1)$. This is the event that the number of upsets so far is small and the majority of agents so far have taken the correct action.

Since there are at most $\beta\log{t}$ upsets, there are at most $\frac{1}{2}\beta\log{t}$ maximal good runs. Since, furthermore, there are at least $\frac{1}{2}t$ agents who take action $+1$, there is at least one maximal good run of length at least $t /(\beta \log{t})$.

Thus, $\mathbb{P}_+(B_1)$ is bounded from above by the probability that there are some $s_1 < s_2 < t$ such that there is a good run of length $s_2-s_1 \geq t/(\beta\log{t})$ from $s_1$ and $a_{s_2} = -1$.

For fixed $s_1$, $s_2$, denote by $E_{s_1, s_2}$ the event that there is a good run of length $s_2 - s_1$ from $s_1$. Denote by $\Gamma_{s_1, s_2}$ the event $(E_{s_1, s_2}, a_{s_2} = -1)$. Then
\begin{align*}
    \mathbb{P}_{+}(\Gamma_{s_1, s_2}) &= \mathbb{P}_{+}(a_{s_2} = -1 | E_{s_1, s_2}) \cdot \mathbb{P}_{+}(E_{s_1, s_2})\\
    &\leq \mathbb{P}_{+}(a_{s_2} = -1 | E_{s_1, s_2}).
\end{align*}
By Lemma~\ref{lem:quickly_same_as_l_star}, there exists a $z > 0$ such that $E_{s_1, s_2}$ implies that $\ell_{s_2} \geq \ell_{s_2 - s_1 - z}^*$. Therefore,
\begin{align*}
    \mathbb{P}_{+}(\Gamma_{s_1, s_2}) &\leq G_+(-\ell_{s_2 - s_1 - z}^*).
\end{align*}
Since for $t$ sufficiently large $\ell_t^* >  t^{\frac{1}{k+2}}$ and since $G_+(-x) \leq \ee^{-x}$ by \eqref{eq:nu},
\begin{align*}
    \mathbb{P}_{+}(\Gamma_{s_1, s_2}) \leq \ee^{-\alpha (s_2 - s_1 - z)^{\frac{1}{k+2}}} \leq \ee^{-\alpha (t/(\beta\log{t}) - z)^{\frac{1}{k+2}}}.
\end{align*}
To simplify, we further bound this last expression to arrive at, for some $c > 0$,
\begin{align*}
    \mathbb{P}_{+}(\Gamma_{s_1, s_2}) \leq c\ee^{-t^{\frac{1}{k+3}}}
\end{align*}
for all $t$.
Since $B_1$ is covered by fewer than $t^2$ events of the form $\Gamma_{s_1, s_2}$ (as $s_1$ and $s_2$ are less than $t$), it follows that
\begin{align*}
    \mathbb{P}_{+}(B_1) < ct^{2}\ee^{-t^{\frac{1}{k+3}}} < \frac{1}{t^{2.1}}
\end{align*}
for all $t$ large enough.

Finally we bound $\mathbb{P}_+(B_2)$. This is the event that the number of upsets so far is small and the majority of agents so far have taken the wrong action. As in $B_1$, there is a maximal bad run of length at least $t/(\beta\log(t))$.

Denote by $R$ the event that there is at least one bad run of length $t/(\beta\log(t))$ before time $t$ and by $R_{s}$ the event that agents $s$ through $s+t/(\beta\log{t})-1$ take action $-1$. Since $B_2$ is contained in $R$, and since $R$ is contained in the union $\cup_{s=1}^t R_s$, we have that
\begin{align*} 
    \mathbb{P}_+(B_2)
    \leq \mathbb{P}_+(R)
    \leq \sum_{s=1}^t\mathbb{P}_+(R_{s}).
\end{align*}
Taking the maximum of all the addends in the right hand side, we can further bound the probability of $B_2$:
\begin{align*}
    \mathbb{P}_+(B_2)\leq t \cdot \max_{1 \leq s \leq t}{\mathbb{P}_+(R_s)}.
\end{align*}
Conditioned on $\ell_s$, the probability of $R_s$ is
\begin{align*}
    \mathbb{P}_+(R_s\,|\,\ell_s) = \prod_{r=s}^{s+t/(\beta\log t)-1}G_+(-\ell_r).
\end{align*}
By Lemma~\ref{lem:ratio to ell star bounded}, there exists $M > 0$ such that $|\ell_{r}| \leq M\ell_t^*$, for all  $r\leq t$. Therefore, since  $G_+$ is monotone, 
\begin{align*}
    \mathbb{P}_+(R_s) \leq G_+(M\ell_t^{*})^{t/(\beta\log t)}.
\end{align*}
It follows that
\begin{align*}
    \mathbb{P}_+(B_2) \leq t \cdot G_+(M\ell_t^{*})^{t/(\beta\log t)}.
\end{align*}
Since $G_+(x) = 1-c\cdot x^{-k}$ for $x$ large enough, and since $\ell_t^*$ is asymptotically at most $t^{1/(k+0.5)}$, we have that 
\begin{align*}
    \log{G_+(M\ell_t^{*})} \leq -c M^{-k}\cdot t^{-k / (k+0.5)}.
\end{align*}
Thus
\begin{align*}
    \mathbb{P}_+(B_2) \leq t \cdot \exp\left(-c M^{-k} \cdot t^{1/(2k+1)} / (\beta\log t)\right) \leq t^{-2.1},
\end{align*}
for all $t$ large enough.
This concludes the proof, because $\mathbb{P}_{+}(a_t = -1) = \mathbb{P}_{+}(A)+\mathbb{P}_{+}(B_1)+\mathbb{P}_{+}(B_2)\leq \kappa\frac 1 {t^{2.1}}$ for some constant $\kappa$.

\end{proof}

Given this bound on the probability of mistakes, the proof of the main theorem of this section follows easily from Lemma~\ref{lem:mistake-bounds-time-to-learn}.
\begin{proof}[Proof of Theorem~\ref{thm:polytail-finite-expectation}]

By Proposition~\ref{proposition:polynom_probability}, there exists $\kappa > 0$ such that $\mathbb{P}(a_t = -1\,|\,\theta = +1) < \kappa \frac{1}{t^{2.1}}$ for all $t \geq 1$. Hence, by Lemma~\ref{lem:mistake-bounds-time-to-learn} $\mathbb{E}(T_L\,|\,\theta = +1)<\infty$. By a symmetric argument the same holds conditioned on $\theta=-1$. Thus, the expected time to learn is finite. 
\end{proof}

\end{document}